\documentclass[12pt, a4paper, reqno]{amsart}
\usepackage[russian,english]{babel}
\usepackage{amsmath,amsfonts,amssymb,amsthm,graphics,epsfig}
\usepackage{hyperref}
\usepackage{graphicx,graphics}
\usepackage{epstopdf}
\usepackage{color}

\usepackage[margin=2.5cm]{geometry}

\newtheorem{thm}{Theorem}[section]
\newtheorem{lem}[thm]{Lemma}
\newtheorem{cor}[thm]{Corollary}
\newtheorem{con}{Conjecture}
\theoremstyle{definition}
\newtheorem{defin}[thm]{Definition}
\newtheorem{exam}[thm]{Example}
\theoremstyle{remark}
\newtheorem*{rem}{Remark}

\newcommand { \ib }[1] {\textit{\textbf{#1}}}

\newcommand { \lin }{\mathop{\rm{lin}}\nolimits}

\begin{document}
\renewcommand{\ib}{\mathbf}
\renewcommand{\proofname}{Proof}
\renewcommand{\phi}{\varphi}
\makeatletter \headsep 10 mm \footskip 10 mm

\title[]{The Voronoi conjecture for parallelohedra with simply connected $\delta$-surface}
\author{A.~Garber}
\author{A.~Gavrilyuk}
\author{A.~Magazinov}

\address{The University of Texas at Brownsville, Department of Mathematics, One University blvd., Brownsville, TX, 78520, USA.}
\email{alexeygarber@gmail.com}

\address{Steklov Mathematical Institute, 8 Gubkina str., Moscow, 119991, Russia, and Center for Optical Neural Technologies, Scientific-Research Institute for System Analysis, RAS, Vavilova str., 44, korp. 2, Moscow 119333, Russia.}
\email{agavrilyuk.research@gmail.com}

\address{Steklov Mathematical Institute, 8 Gubkina str., Moscow, 119991, Russia.}
\email{magazinov-al@yandex.ru}

\date{\today}

\begin{abstract}
We show that the Voronoi conjecture is true for parallelohedra with simply connected $\delta$-surface. Namely, we show that if the boundary of parallelohedron $P$ remains simply connected after removing closed non-primitive faces of codimension 2, then $P$ is affinely equivalent to a Dirichlet-Voronoi domain of some lattice. Also we construct the $\pi$-surface associated with a parallelohedron and give another condition in terms of homology group of the constructed surface. Every parallelohedron with simply connected $\delta$-surface also satisfies the condition on homology group of the $\pi$-surface.
\end{abstract}

\maketitle

\section{The Voronoi conjecture}

\begin{defin}
A $d$-dimensional polytope $P$ is called a {\it parallelohedron} if $\mathbb{R}^d$ can be tiled by non-overlapping translates of $P.$ 
\end{defin}
  
A tiling by parallelohedra is called {\it face-to-face} if the intersection of any two copies of $P$ is a face of both, and is called a {\it non face-to-face} otherwise. 
B.~Venkov \cite{Ven} and later independently P.~McMullen \cite{McM} proved that if there is a non face-to-face tiling by $P$, then there is also a face-to-face tiling. 
It is clear that the face-to-face tiling by $P$ is unique up to translation.
We will denote this tiling as $\mathcal{T}(P)$ or just $\mathcal{T}$ when the generating polytope of the tiling is obvious.


In 1897 H.~ Minkowski \cite{Min} proved that every parallelohedron $P$ is centrally symmetric, and all facets of $P$ are centrally symmetric. Later Venkov \cite{Ven} added the third necessary condition, he proved that the projection of $P$ along any face of codimension $2$ is a two-dimensional parallelohedron, i.e., a parallelogram or a centrally symmetric hexagon. Also Venkov proved that these three conditions are sufficient for a convex polytope to be a parallelohedron. In 1980 McMullen \cite{McM} gave an independent proof that these three conditions are necessary and sufficient, see also \cite{McM2} for acknowledgment of priority.



The centers of all tiles of $\mathcal{T}(P)$ form a $d$-dimension lattice $\Lambda(P)$. If the fundamental domain of $\Lambda(P)$ has 
volume 1, then the homothetic polytope $2P$ is centrally symmetric with respect to the origin, has volume $2^d$, and contains no lattice points in the interior. This, by definition, means that $2P$ is an {\it extremal body}. Additional information about extremal bodies can be found in \cite[Ch.2 \S 12]{GL}.

On the other hand, given a $d$-dimensional lattice $\Lambda$, the Dirichlet-Voronoi polytope $P(\Lambda)$ is a parallelohedron; the Dirichlet-Voronoi polytope is the set of points  that are closer to a given lattice point $O$ than to any other point of $\Lambda$. 

\begin{con}[G.~Voronoi, \cite{Vor}]\label{vorcon}
Any $d$-dimensional parallelohedron $P$ is affinely equivalent to a Dirichlet-Voronoi polytope $P(\Lambda')$ for some $d$-dimensional lattice $\Lambda'$.
\end{con}

Voronoi's conjecture has been proved for several families of parallelohedra with special local combinatorial properties.

Denote the set of all $k$-faces of a tiling $\mathcal{T}$ by $\mathcal{T}^k$, and the set of all $k$-faces of a polytope $P$ by $P^k$. 

\begin{defin}
A face $F\in \mathcal{T}^{d-k}$ is called {\it primitive} if $F$ belongs to exactly $k+1$ tiles in $\mathcal{T}$. 
\end{defin}

For example, any facet of $\mathcal{T}$ is a primitive because it belongs to exactly two tiles. For a cubic tiling of $\mathbb{R}^d$ the facets are the only primitive faces. 

A face $F$ of a parallelohedron $P$, with codimension $2$, is primitive if projection of $P$ along $F$ is a hexagon, and is otherwise non-primitive; in the non-primitive case $F$ belongs to four tiles in $\mathcal{T}(P)$. The set of facets of $P$ that are parallel to a given face $F$ with codimension $2$ is called a {\it belt} of $P$. A belt can have $4$ or $6$ facets and accordingly is called {\it $4$-} or {\it $6$-belt}. 

\begin{figure}[!ht]
\begin{center}
\includegraphics[scale=0.9]{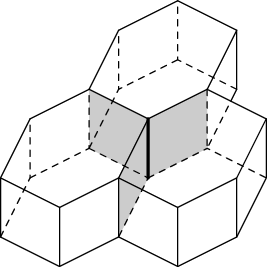}
\hskip 3cm
\includegraphics[scale=0.9]{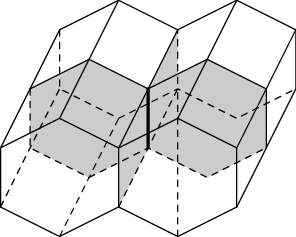}
\end{center}

\caption{Local structure of primitive and non-primitive $(d-2)$-faces.}
\label{dual2cells}
\end{figure}

\begin{defin}
A parallelohedron $P$ is called {\it $k$-primitive} if all the  $k$-faces of $\mathcal{T}(P)$ are primitive. A $0$-primitive parallelohedron is usually called a {\it primitive} parallelohedron. 
\end{defin}

Clearly, any $k$-primitive parallelohedron is also $(k+1)$-primitive. 

In 1909 Voronoi \cite{Vor} (see also \cite[Ch.2 \S 12]{GL}) proved Conjecture \ref{vorcon} for primitive parallelohedra 
In 1929 Zhitomirskii \cite{Zhit} proved Voronoi's conjecture for $(d-2)$-primitive parallelohedra 
In this case all the projections along $(d-2)$-faces are hexagons. 

\begin{defin}\label{kirred}
Assume that $k>1$, and that $G\in \mathcal{T}^{d-k}(P)$. Let $\mathcal{N}_G$ be the set normal vectors for all facets containing $G$. Then $P$ is {\it $k$-irreducible} if for each $G\in \mathcal{T}^{d-k}(T)$ the set $\mathcal{N}_G$ cannot be represented as $\mathcal{N}_G=\mathcal{N}_1\cup\mathcal{N}_2$ where $\mathcal{N}_1$, $\mathcal{N}_2$ are non-empty, and $\lin \mathcal{N}_1\cap \lin \mathcal{N}_2=\{\mathbf{0}\}.$
\end{defin}


For example, the result of Zhitomirskii \cite{Zhit} establishes Voronoi's conjecture for $2$-irreducible parallelohedra.

In 2005 A.~Ordine \cite{Ord} proved Voronoi's conjecture for $3$-irreducible parallelohedra. Up to the moment, no improvements of this result are known.

We also mention a result by R.~Erdahl \cite{Erd} who proved Voronoi's conjecture for space-filling zonotopes. A {\it zonotope} is a Minkowski sum of finitely many segments. Erdahl's proof is based on technique of unimodular vector representations, and he constructed the appropriate affine transformation for a zonotope using normals to its facets. This approach significantly differs from the original approach introduced by Voronoi and developed later by Zhitomirskii and Ordine where the main focus is on constructing an auxiliary polyhedral surface named generatrissa and building a quadratic form associated to the surface.


\begin{defin}\label{deltasur}
The surface $\partial P$ of a $d$-dimensional parallelohedron $P$ is homeomorphic to the $(d-1)$-dimensional sphere $\mathbb{S}^{d-1}$. After deletion of all closed non-primitive faces of codimension 2 of $P$ we obtain new $(d-1)$-dimensional manifold without boundary (in the topological sense of notion of ``manifold without boundary''). We call this manifold the $\delta$-{\it surface} of $P$ and denote it by $P_\delta$.

The manifold $P_\delta$ is centrally symmetric because $\partial P$ is centrally symmetric. 
If we glue together every pair of opposite points of the $P_\delta$, then we obtain another $(d-1)$-dimensional manifold that is a subset of real projective space $\mathbb{RP}^{d-1}$. We call this manifold {\it $\pi$-surface} of $P$ and denote by $P_\pi$.



\end{defin}

The $\delta$-surface is related to notions of Venkov graph and Venkov subgraph of a parallelohedron. Consider a graph with vertices corresponding to pairs of the antipodal (opposite) facets of parallelohedron $P$. Edges connect any two different vertices corresponding to four facets of $P$ belonging to a single belt. The edges a colored in blue and red. If the facets belong to a $4$-belt then the edge is blue, otherwise they belong to a $6$-belt and the edge is red. The graph is called {\it Venkov graph} of parallelohedron $P$. Subgraph with the same set of vertices and just red edges remained is called the {\it red Venkov subgraph} of $P$ (for details see \cite{Gavr}).

A. Ordine proved a theorem \cite[Thm. 2]{Ord} that a parallelohedron $P$ is irreducible if and only if its red Venkov subgraph is connected. The ``if'' part is easy. For the ``only if'' he used the technique of gain functions and (additive) version of the quality translation theorem by Ryshkov, Rybnikov \cite{RR}. As the first step Ordine proved that connectivity components of the red Venkov subgraph of $P$ correspond to independent sublattices in lattice consist of centers of polytopes of the tiling $\mathcal{T}(P)$. Second, he showed that a union of parallelohedra centered in points of such a sublattice is convex and invariant under translations of the sublattice affine hull. Essentially that finishes the proof. The preprint \cite{MO} contains a rewritten proof of Ordine's result, but it is in Russian.

The theorem is equivalent to a condition that $P$ is irreducible iff its boundary remains connected after deletion of all non-primitive faces of codimension 2.
Thus the $\delta$-surface of $P$ is connected iff $P$ cannot be represented as a direct sum of two parallelohedra of lower dimensions. The same holds for $\pi$-surface $P_\pi$. Ordine also showed that the general case follows if Voronoi's conjecture is proved for the irreducible case where a parallelohedron can not be represented as a direct sum. Therefore, we restrict our attention to parallelohedra with connected $\delta$-surface. 
If a parallelohedron can be represented as direct sum of two parallelohedra of lower dimensions, that is what we called reducible, then it is reducible at every vertex in the sense of the Definition \ref{kirred}.


In this paper we will prove (Theorem \ref{simpconn}) Voronoi's conjecture for parallelohedra with simply connected $\delta$-surface. Also we prove (Theorem \ref{vorgener}) Voronoi's conjecture for a family of parallelohedra with first $\mathbb Q$-homology group of $\pi$-surface generated by ``trivial'' cycles. These conditions are global while all preceding conditions which using canonical scaling method of Voronoi (see Section \ref{sect:cansc}) were local. Our results generalize theorems of Voronoi and Zhitomirskii but for now it is unclear whether our theorems are applicable to Ordine's case or not.

However, the comparison of Ordine's conditions and conditions of Theorem \ref{vorgener} for irreducible parallelohedra in small dimensions can be done. For reducible parallelhedra $P=P_1\oplus P_2$ Voronoi's conjecture is inherited from Voronoi conjecture for Minkowski summands $P_1$ and $P_2$. For dimension 3 this is done in the Section \ref{3dim}, and both conditions are true for all three possible irreducible three-dimensional parallelohedra. The four-dimensional case is done in \cite{Gar} and it is shown there that all four-dimensional parallelohedra satisfies conditions of the Theorem \ref{vorgener}. Using the methods described in \cite[Sect. 5]{Gar} for space-filling zonotopes (i.e. a zonotope and parallelohedron at the same time), one can show that the five-dimensional zonotope with zone vectors (i.e. vectors representing corresponding segments in the Minkowski sum, see also \cite{McM0} for the definition) represented by columns of the matrix below does not satisfy Ordine's conditions. It is a $\Pi$-zonotope (see the same paper \cite[Sect. 5]{Gar}) with the graph on six vertices with seven edges forming a 6-cycle with a long diagonal.

$$\left(
\begin{array}{rrrrrrr}
1&0&0&-1&0&0&0\\
0&0&0&1&-1&0&0\\
0&1&0&0&1&-1&0\\
0&0&0&0&0&1&-1\\
0&0&1&0&0&0&1
\end{array}
\right).$$

For three- and four-dimensional cases (Section \ref{3dim} and \cite{Gar}) the Theorem \ref{vorgener} is equivalent to the Voronoi's conjecture, but we don't know any proof for such equivalence for arbitrary dimension, or at least for five-dimensional parallelohedra. Also we don't know any example of a parallelohedron that satisfies the Voronoi conjecture and does not satisfy the Theorem \ref{vorgener}.

It seems, that ideas of topological extension of Voronoi's method were developed before, but we did not find any published article or available preprint. As a reference we can give the link \cite{Die} where in the last three paragraphs the result of oral communication of T. Dienst and and son of B.A. Venkov on connection of Voronoi's conjecture and certain cohomology group is exposed. In particular, B.A. Venkov told, that his father ``showed that for a given parallelohedron, the conjecture follows from the vanishing of a certain cohomology group assigned to this polytope. By this, he was able to prove the conjecture for all polytopes with no more than one belt of length 4'' (this is a precise quote from the web-site \cite{Die}).

\section{Canonical scaling and surface of a parallelohedron}\label{sect:cansc}

This section is devoted to the notion of a canonical scaling and its connection with Voronoi's conjecture.

\begin{defin}
A function $\mathfrak{s}:\mathcal{T}^{d-1}\longrightarrow \mathbb{R}_+$ is called a {\it canonical scaling} for the tiling $\mathcal{T}$ (for a parallelohedron $P$) if for each $(d-2)$-dimensional face $G\in \mathcal{T}^{d-2}$ the direction of the unit normal $\ib n_i$ for the facets $F_i$ that contain $G$ can be chosen so that
\begin{equation}
\sum\limits_{i=1}^{3 \text{ or }4}\mathfrak{s}(F_i)\ib n_i=\ib 0;\label{csdef}
\end{equation}
the sum ranges to $3$ if $G$ is primitive, and to $4$ if $G$ is not primitive.
\end{defin}

\begin{lem}\label{conjequiv}
Voronoi's conjecture is true for a parallelohedron $P$ iff the corresponding tiling $\mathcal{T}(P)$ admits a canonical scaling.
\end{lem}

\begin{rem}
This property was used by Voronoi \cite{Vor}, Zhitomirskii \cite{Zhit}, and Ordine \cite{Ord} to prove Voronoi's conjecture \ref{vorcon} for particular classes of parallelohedra. For a proof of Lemma \ref{conjequiv} see \cite{DG, CS, RR, Vor}.
\end{rem}

A primitive $(d-2)$-face $G$ belongs to exactly three facets, $F_1, F_2, F_3$. Unit normals $\ib n_i$ to these facets span a $2$-dimensional plane since they are orthogonal to $G$ and no two of them are collinear, so there is exactly one (up to a nonzero factor) linear dependence
$$\alpha_1\ib n_1+\alpha_2\ib n_2+\alpha_3\ib n_3=\ib 0.$$
Therefore, a canonical scaling $\mathfrak{s}$ for $\mathcal{T}$ (if exists) must satisfy the local rule:
\begin{equation}\label{csrule}
\dfrac{\mathfrak{s}(F_i)}{\mathfrak{s}(F_j)} = \dfrac{|\alpha_i \mathbf n_i|}{|\alpha_j \mathbf n_j|} = \left|\dfrac{\alpha_i}{\alpha_j}\right| .
\end{equation}

\begin{defin}
The fraction $\left|\dfrac{\alpha_j}{\alpha_i}\right|$ is the value of the {\it gain function} $\mathfrak{g}(F_i,F_j)$, defined on pairs of facets that share a primitive face of codimension two. 

As we have just seen, the gain function is uniquely defined for pairs of facets linked by a primitive $(d-2)$-face. The gain function shows how a canonical scaling changes along a path that travels facet-to-facet across such primitive faces. The gain function has the property that $\mathfrak{g}(F_i,F_j)=\dfrac{1}{\mathfrak{g}(F_j,F_i)}$.
\end{defin}

In the remaining portion of this section we use the gain function to obtain necessary and sufficient conditions for the existence of a canonical scaling.

\begin{lem}\label{loccansc}
If there exists a positive valued function $\mathfrak{s}'$ on the set of all facets of $P$ that satisfies condition $\eqref{csrule}$ for every two facets with common primitive $(d-2)$-face then there exists a canonical scaling $\mathfrak{s}:\mathcal{T}^{d-1}\longrightarrow \mathbb{R}_+$.
\end{lem}

\begin{rem}
The inverse statement is trivial since we can take as $\mathfrak{s}'$ the restriction of $\mathfrak{s}$ on the surface of one copy of $P$.
\end{rem}

\begin{proof}
Consider an arbitrary facet $F$ and its opposite facet $F'$ of $P$.

If $F$ belongs to some 6-belt then application of rule \eqref{csrule} to this belt immediately gives us $\mathfrak{s}'(F)=\mathfrak{s}'(F').$ If $F$ does not belong to any $6$-belt, and $\mathfrak{s}'(F')\neq \mathfrak{s}'(F)$, then set $\mathfrak{s}'(F'):=\mathfrak{s}'(F)$. This modification of $\mathfrak{s}'$ preserves the hypothesis that condition \eqref{csrule} holds at all primitive $(d-2)$-faces.

Now the function $\mathfrak{s}'$ is invariant with respect to central symmetry of $P.$ We translate this function on all copies of $P$ in the tiling $\mathcal{T}$. This translation is correctly defined since different tiles of $\mathcal{T}$ are glued together by opposite facets of $P$ and values of $\mathfrak{s}'$ are equal on these facets. The constructed function $\mathfrak{s}:\mathcal{T}^{d-1}\longrightarrow \mathbb{R}_+$  satisfies condition \eqref{csdef} for every primitive $(d-2)$-face because $\mathfrak{s}'$ satisfies rule \eqref{csrule}. For a non-primitive $(d-2)$-face condition~\eqref{csdef}  holds as well because every such face lies in two pairs of opposite facets with equal values of $\mathfrak{s}'$ in each pair. Consequently, we can choose opposite normal directions for the facets in each pair and obtain a zero sum. Thus the function $\mathfrak{s}$ is a canonical scaling.
\end{proof}

\begin{defin} \label{primcycle}
A sequence of facets $\gamma=\left[F_0,\ldots,F_k\right]$ is a {\it primitive combinatorial path} on $P$ if consecutive facets $F_i$ and $F_{i+1}$ are linked by a common primitive face of codimension 2. We call $\gamma$ a {\it primitive cycle} if $F_0=F_k$.
\end{defin}

Define the gain function $\mathfrak{g}$ for every primitive path of $P$ by the formula 
$$\mathfrak{g}(\gamma)=\prod_{i=1}^k\mathfrak{g}(F_{i-1},F_i).$$

We call a curve $\gamma$ on the $\delta$-surface of $P$ (see Definition \ref{deltasur}) {\it generic}: if endpoints of $\gamma$ are interior to facets of $P$; $\gamma$ does not intersect any face of dimension less than $d-2$; and the intersection of $\gamma$ with a $(d-2)$-dimensional face of $P$ is transversal.

For every generic curve $\gamma$ on the $\delta$-surface of $P$ we define a {\it supporting primitive combinatorial path} $\gamma'$ consisting of facets that support $\gamma$. 
This allows us to define a gain function on generic curves: $\mathfrak{g}(\gamma):=\mathfrak{g}(\gamma')$. Obviously, for the union of two curves $\gamma=\gamma_1\cup \gamma_2$ we have $\mathfrak{g}(\gamma)=\mathfrak{g}(\gamma_1)\cdot \mathfrak{g}(\gamma_2)$.

\begin{lem}\label{iffcond} Voronoi's conjecture $\ref{vorcon}$ is true for a parallelohedron $P$ iff $\mathfrak{g}(\gamma)=1$ holds for every generic cycle $\gamma$ on the $\delta$-surface of $P$.
\end{lem}
\begin{proof}
Assume that $\mathfrak{g}(\gamma)=1$ holds for every generic cycle $\gamma$. We will construct a function $\mathfrak{s}'$ from Lemma \ref{loccansc} in the following way. Consider an arbitrary facet $F$ and put $\mathfrak{s}'(F)=1$. Now for every facet $G$ consider an arbitrary generic curve $\gamma$ that starts in the center of $F$ and ends in the center of $G$ and put $\mathfrak{s}'(G)=\mathfrak{g}(\gamma)$. We show that $\mathfrak{s}'$ is defined correctly and satisfies conditions of Lemma \ref{loccansc}.

Assume that two different curves $\gamma_1$ and $\gamma_2$ produce different values of $\mathfrak{s}'$ on a facet $G$. Then $\mathfrak{g}(\gamma_1)\neq\mathfrak{g}(\gamma_2)$. For a cycle $\gamma=\gamma_1\cup\gamma_2^{-1}$ (here $\gamma_2^{-1}$ denotes the reversed curve $\gamma_2$) we have $\mathfrak{g}(\gamma)=\dfrac{\mathfrak{g}(\gamma_1)}{\mathfrak{g}(\gamma_2)}\neq 1$, a contradiction. So the functions $\mathfrak{s}'$ is defined correctly on the set of all facets of $P$.

In order to show that $\mathfrak{s}'$ satisfies conditions of the Lemma \ref{loccansc}
consider two arbitrary facets $F_1$ and $F_2$ with common non-primitive $(d-2)$-face. Assume that values on facets $F_i$ were obtained with paths $\gamma_i$, consider a path $\gamma_3$ that connects centers of $F_1$ and $F_2$ through their common $(d-2)$-face. We have the cycle $\gamma=\gamma_3\cup\gamma_2^{-1}\cup\gamma_1$, so $$\mathfrak{g}(\gamma_3)=\dfrac{\mathfrak{g}(\gamma_2)}{\mathfrak{g}(\gamma_1)}=\frac{\mathfrak{s}'(F_2)}{\mathfrak{s}'(F_1)}.$$ 
By definition, $\mathfrak{g}(\gamma_3)=\mathfrak{g}(F_1,F_2)$, therefore $\mathfrak{s}'$ satisfies condition of Lemma \ref{loccansc} and there exists a canonical scaling of the tiling, so by lemma \ref{conjequiv} Voronoi's conjecture is true for $P$.

On the other hand, if Voronoi's conjecture is true for $P$ then there is a canonical scaling $\mathfrak{s}:\mathcal{T}^{d-1}\longrightarrow \mathbb{R}_+$, from which we can determine the corresponding gain function. It easily follows using \eqref{csrule} that $\mathfrak{g}(\gamma)=1$ for every generic cycle $\gamma$ on the $\delta$-surface of $P$.
\end{proof}

\section{Dual 3-cells and local consistency of canonical scaling}

\begin{defin}
Consider an arbitrary face $G$ of codimension $k$ of the tiling $\mathcal{T}(P).$ Then, the corresponding {\it dual $k$-cell} $\mathcal{D}_G$ is the convex hull of the centers of all tiles that share $G$. 

For example, the dual cell corresponding to a $d$-dimensional polytope $P'\in \mathcal{T}(P)$ is a single point --- its center.
\end{defin}

\begin{defin}
We say that a dual cell $\mathcal{D}_G$ {\it is a face} of a dual cell $\mathcal{D}_H$ iff $H$ is a face of $G$. Using that we can introduce the notion of {\it incidence} for dual cells. A cell $\mathcal{D}_G$ is {\it incident} to a dual cell $\mathcal{D}_H$ iff $H$ is incident to $G$.
\end{defin}

It is clear that a face $\mathcal{D}_G$ of a dual cell $\mathcal{D}_H$ belongs to it as a point-set, but it is not proved and no counterexample known whether $\mathcal{D}_G$ is a polyhedral face of $\mathcal{D}_H$.
Such defined incidences on dual cells induce an inverse face lattice structure comparing with the face lattice of $\mathcal{T}(P)$. 

\begin{defin}
$\mathcal{D}_F$ is said to be {\it combinatorially equivalent} to a polytope in case the polytope has identical face lattice. 

For example, a dual $2$-cell $\mathcal{D}_F$ is combinatorially equivalent to either a triangle for primitive $F$ or to a parallelogram for a non-primitive $F$ (see figure \ref{dual2cells}).
\end{defin}

We will use the following classification theorem on $(d-3)$-faces of parallelohedra, or equivalently dual 3-cells. The first part of the theorem was originally
stated by Delone, and the second is a reformulation in terms of dual cells.

\begin{thm}[Delone, \cite{Del}, see also \cite{Ord}]\label{thmdel}
There are five possible combinatorial types of coincidence of parallelohedra at $(d-3)$-faces, and each of five types (see figure $\ref{dual3cells}$) has a 
representative when $d = 3$.

In other words, every dual $3$-cell is combinatorially equivalent to one of the five $3$-dimensional polytopes: cube, triangular prism, tetrahedron, octahedron, or quadrangular pyramid.
\end{thm}

\begin{figure}[!ht]
\begin{center}
\includegraphics[scale=0.7]{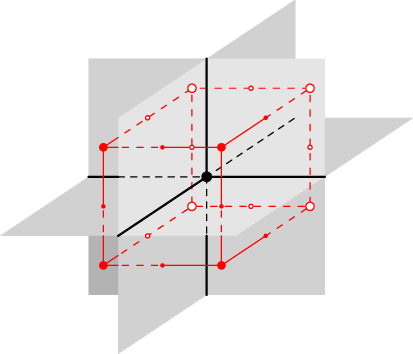}
\hskip 1cm
\includegraphics[scale=0.7]{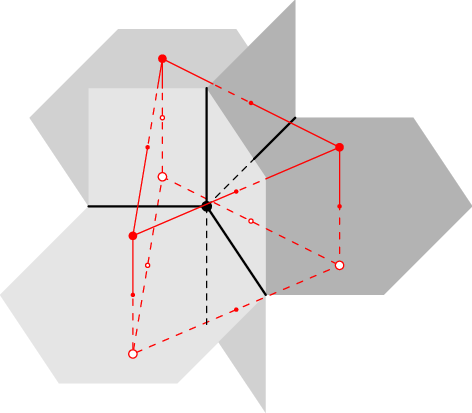}
\hskip 1cm
\includegraphics[scale=0.7]{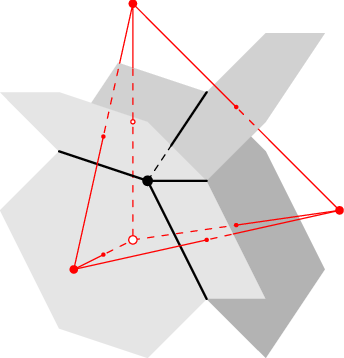}
\hskip 1cm
\includegraphics[scale=0.7]{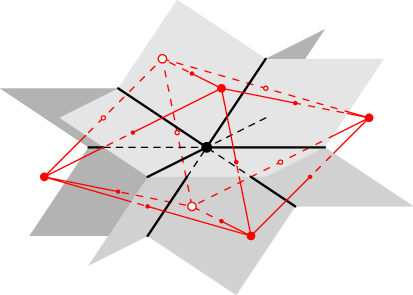}
\hskip 1cm
\includegraphics[scale=0.7]{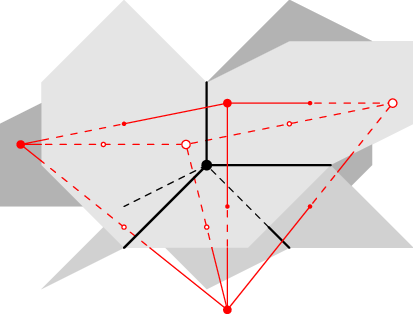}

\end{center}

\caption{Dual $3$-cells.}
\label{dual3cells}
\end{figure}

Every $(d-3)$-face $G$ of a $d$-polytope $P$ is a face of codimension $2$ in the boundary complex $\partial P$ which is homeomorphic to a $(d-1)$-sphere. We consider a union $U_G$ of all facets $F_i$ of $P$ containing $G$ as a neighborhood of $G$ in $\partial P$. $U_G$ is naturally homeomorphic to a product $G\times D^2$ where $D^2$ is a $2$-dimensional disk. Each facet $F_i$ containing $G$ could be coherently represented as $G\times S_i \subset G\times D^2$ where $S_i$ is a $2$-dimensional sector of $D^2$. Union of sectors is the entire $D^2$. Thus all the facets $F_i$ containing $G$ form a closed combinatorial path where each $F_i$ appears just once along it. Clearly, the closed path is unique up to two choices of direction for a bypass along it. 

\begin{defin}
The closed combinatorial path generated by $(d-3)$-face $G$ of a $d$-polytope $P$ on the surface of $P$, as described above, we call a {\it cycle around} $G$.
\end{defin}

\begin{lem}\label{local}
Let $G$ be a $(d-3)$-face of a parallelohedron $P$. If the $(d-2)$-faces containing $G$ are all primitive, then, the cycle around $F$ is primitive, and the gain function along this path has value $1$.
\end{lem}

\begin{rem}
The cycle around $G$ is a primitive cycle in sense of Definition \ref{primcycle}.
\end{rem}

In the proof of Lemma~\ref{local} we will use the following Lemma.

\begin{lem}\label{zh}
Let $G$ be a $(d-3)$-face of parallelohedron $P$. Suppose that all $(d-2)$-faces of $P$ containing $G$ are primitive.
Let $P = P_0$, $P_1$, $P_2, \ldots$ be all the parallelohedra of $\mathcal T$ that share $G$. Then there exist 
affine functions $U_0, U_1, U_2, \ldots$ with the following property: given any pair of parallelohedra $(P_i, P_j)$
with a common facet, the functions $U_i$ and $U_j$ coincide on the affine hull of the facet $P_i\cap P_j$ and nowhere else.
\end{lem}

\begin{rem}
This lemma is an extension of a statement by Zhitomirskii~\cite[Sect. 1 and 4]{Zhit}. He needed all $(d-2)$-faces of the entire tiling containing $G$ to be primitive. We relax the condition and demand just part of these $(d-2)$-faces to be primitive, namely those belonging to $P$. As a result we admit an extra case $(3)$ for a dual cell of $G$.
\end{rem}

\begin{proof} 
Theorem \ref{thmdel} implies that there are three possible cases.

\begin{enumerate}
	\item The dual cell of $G$ is combinatorially equivalent to a tetrahedron.
	\item The dual cell of $G$ is combinatorially equivalent to an octahedron.
	\item The dual cell of $G$ is combinatorially equivalent to a quadrangular pyramid and $P$ corresponds to its apex.
\end{enumerate}

We will construct $U_i$ to be constant on every $(d-3)$-dimensional affine plane parallel to $G$. Therefore we will restrict ourselves to
the 3-dimensional images of parallelohedra $P_i$ under the projection $\pi_G$ along the face $G$.

The projections $\{\pi_G(P_i) : i = 0, 1, 2, \ldots\}$ split the neighbourhood of the point $\pi_G(G)$ in the same way as a 3-dimensional polyhedral fan
$\mathcal C$ does. The combinatorics of $\mathcal C$ is completely prescribed by the type of the dual cell of $G$.

Our goal is, in fact, to construct a piecewise affine continuous function $U$ on $\mathcal C$ such that the restriction of $U$ to each 3-dimensional
cone of $\mathcal C$ is affine function, and the affine functions for different cones are different. For Cases 1 and 2 this was done by Zhitomirskii, however,
in order to make our proof more transparent, we will consider these cases again along with Case 3. For simplicity, suppose that $\pi_G(G)$ is
the origin. We will also require that $U(\pi_G(G)) = 0$, therefore we can speak of linear functions rather than affine.

\noindent{\bf Case 1.} $\mathcal C$ has four 1-dimensional faces (rays). Let $\mathbf x_1$, $\mathbf x_2$, $\mathbf x_3$, $\mathbf x_4$ be the unit vectors
along these rays. For every $\lambda \geq 0$ and $j = 1, 2, 3, 4$ set $U(\lambda \mathbf x_j) = \lambda$. It defines $U$ because
every 3-dimensional cone $C$ of the fan $\mathcal C$ is 3-sided, and three linear functions on extreme rays of $C$ can be extended in a unique way to a 
linear function on $C$.

\noindent{\bf Case 2.} Let $C$ be one of the 3-dimensional cones of the fan $\mathcal C$. $C$ is 4-sided, and let $\mathbf x_1$, $\mathbf x_2$, $\mathbf x_3$, $\mathbf x_4$ be the vectors along the extreme rays of $C$ in the order they met by a path around $\pi_G(G)$ on $\partial C$. After a proper scaling of
these vectors, we can assume that $\mathbf x_1 + \mathbf x_3 = \mathbf x_2 + \mathbf x_4$. Now for every $\lambda \in \mathbb R$ and $j = 1, 2, 3, 4$ set
\begin{equation}\label{pl}
U(\lambda \mathbf x_j) = |\lambda|.
\end{equation}

Let $C'$ be a 4-sided 3-dimensional cone with the vertex at the origin. Let $\mathbf y_1$, $\mathbf y_2$, $\mathbf y_3$, $\mathbf y_4$ be the vectors along 
its extreme rays such that $\mathbf y_1 + \mathbf y_3 = \mathbf y_2 + \mathbf y_4$. Suppose that $u_i$, $i = 1, 2, 3, 4$, are linear functions on the extreme
rays of $C'$ corresponding to $\mathbf y_i$. One can check that $u_1$, $u_2$, $u_3$ and $u_4$ can be extended to the same linear function $U'$ if and only if
\begin{equation}\label{4s} 
u_1(\mathbf y_1) + u_3(\mathbf y_3) = u_2(\mathbf y_2) + u_4(\mathbf y_4).
\end{equation}

The condition above can be applied to any of the six cones of the fan $\mathcal C$ Therefore the function $U$ can be recovered from its values on all 
1-dimensional faces of the fan $\mathcal C$, if those values are as in formula \eqref{pl}.

\noindent{\bf Case 3.} Let $C$ be the only 4-sided cone of the fan $\mathcal C$. Once again, we can assume that $\mathbf x_1$, $\mathbf x_2$, $\mathbf x_3$, $\mathbf x_4$ are the vectors along the extreme rays of $C$ satisfying $\mathbf x_1 + \mathbf x_3 = \mathbf x_2 + \mathbf x_4$. 
Let $\mathbf y$ be the vector along the remaining 1-dimensional face of $\mathcal C$. Because of the structure of $\mathcal C$, the 2-face of $\mathcal C$
between $\mathbf x_1$ and $\mathbf y$, and the 2-face between $\mathbf x_3$ and $\mathbf y$ lie in the same plane. Therefore the vectors 
$\mathbf x_1$, $\mathbf x_3$, and $\mathbf y$ are coplanar. Similarly, the vectors $\mathbf x_2$, $\mathbf x_4$, and $\mathbf y$ are coplanar as well.
Consequently, $\mathbf y$ is a linear combination of $\mathbf x_1$ and $\mathbf x_3$, and, simultaneously of $\mathbf x_2$ and $\mathbf x_4$. This is possible
only if $\mathbf y$ is collinear to the vector $-\mathbf x_1 - \mathbf x_3 = -\mathbf x_2 - \mathbf x_4$.

Since $\mathcal C$ consists of convex cones, $\mathbf y$ has the same direction with $-\mathbf x_1 - \mathbf x_3$, but not with $\mathbf x_1 + \mathbf x_3$.
Hence the directions of rays of $\mathcal C$ are the directions of vectors $\mathbf x_1$, $\mathbf x_2$, $\mathbf x_3$, $\mathbf x_4$, and 
$-\mathbf x_1 - \mathbf x_3$.

For every $\lambda \geq 0$ and $j = 1, 2, 3, 4$ set $U(\lambda \mathbf x_j) = \lambda$ and $U(-\lambda \mathbf x_1 - \lambda \mathbf x_3) = \lambda$.
Because of the condition \eqref{4s}, we can recover $U$ on the cone $C$. Further, as in Case 1, we can recover $U$ on all other cones of $\mathcal C$,
because these cones are 3-sided.

Thus in every case we have constructed the function $U$. Returning from the projections to the parallelohedra $P_i$, we get the required set of functions
$U_j$.
\end{proof}


\begin{proof}[Proof of Lemma~\ref{local}]

Lemma~\ref{zh} provides us with the set of functions $U_0, U_1, U_2, \ldots$ such that $U_i$ and $U_j$ coincide on
$P_i \cap P_j$, if $P_i \cap P_j$ is a $(d-1)$-face. For all such pairs $(i, j)$ let $U_j = U_i + \mathbf a^T_{ij}\mathbf x + b_{ij}$
for vector argument $\mathbf x$ and some vector parameter $\mathbf a_{ij}$. 

Assume that $P_i$ and $P_j$ share a common facet $F_{ij}$. The difference $U_j - U_i$ is zero on $F_{ij}$ and is not zero outside
the affine hull of $F_{ij}$. Thus $U_j - U_i$ is not a constant and therefore $\mathbf a_{ij} \neq \mathbf 0$.

Hence the affine hull of $F_{ij}$ is a hyperplane represented by the equation $\mathbf a^T_{ij}\mathbf x + b_{ij} = 0$. Since
the vector $\mathbf a_{ij}$ is a normal to the hyperplane, we conclude that $\mathbf a_{ij}$ and $F_{ij}$ are orthogonal. 

If $P_0$, $P_i$ and $P_j$ share a common $(d-2)$-face $F_{0ij}$, then, since $F_{0ij}$ is primitive and
$$\mathbf a_{0i} + \mathbf a_{ij} - \mathbf a_{0j} = \mathbf 0,$$
we have $\mathfrak g(F_{0i}, F_{0j}) = |\mathbf a_{0i}|/|\mathbf a_{0j}|$.

Now if $\gamma = [F_{01}, F_{02}, \ldots F_{0n}, F_{01}]$ is a cycle around $F$, then
$$\mathfrak g(\gamma) = \frac{|\mathbf a_{01}|}{|\mathbf a_{02}|}\cdot \frac{|\mathbf a_{02}|}{|\mathbf a_{03}|}\cdot \ldots \cdot
\frac{|\mathbf a_{0n}|}{|\mathbf a_{01}|} = 1.$$
 
\end{proof}

\section{Voronoi's conjecture for parallelohedra with simply connected $\delta$-surface}

\begin{lem}\label{hom}
If two generic cycles $\gamma_1$ and $\gamma_2$ on the $\delta$-surface of a parallelohedron $P$ are homotopy equivalent then $\mathfrak{g}(\gamma_1)=\mathfrak{g}(\gamma_2)$.
\end{lem}

\begin{rem}
Under ``homotopy'' here and after we mean the relation of continuous homotopy.
\end{rem}

\begin{proof}
Consider an arbitrary homotopy $F(t)$ between $\gamma_1$ and $\gamma_2$ such that $F(0)=\gamma_1$ and $F(1)=\gamma_2.$ With small perturbation of the homotopy $F$ we can obtain another homotopy $G(t)$ such that:

\begin{enumerate}
\item $G(0)=\gamma_1$ and $G(1)=\gamma_2$;

\item at any moment $t$ cycle $G(t)$ does not intersect any face of $P$ with dimension less than $d-3$;

\item at any moment $t$ cycle $G(t)$ does not have more than one point of intersection with set of all $(d-3)$-faces of $P$;

\item there is only a finite number of moments $t_1,\ldots,t_n$ such that each cycle $G(t_i)$ intersects some $(d-3)$-face $F^{d-3}_i$ of $P$;

\item there is only finite number of moments $\tau_1,\ldots,\tau_k~ (t_i\neq \tau_j)$ such that each cycle $G(\tau_j)$ has exactly one non-transversal intersection of $G(\tau_j)$ with some $(d-2)$-face $F^{d-2}_j$. For all other $t\neq \tau_j$ each intersection of $G(t)$ with $(d-2)$-faces of $P$ is transversal.
\end{enumerate}
So in other words we can find a perturbation of the homotopy $F(t)$ that will be in general position.

For any $t \in \left(0, 1\right)$ equal neither to any $t_i$ nor to any $\tau_j$ the cycle $G(t)$ is generic, so we can evaluate the gain function $\mathfrak{g}(G(t))$. We show that $\mathfrak{g}(G(t))$ does not depend on $t$.

For any segment $[a,b]\subset[0,1]$ that does not contain points
$t_i$ and $\tau_j$ the function $\mathfrak{g}(G(t))$ is constant because the primitive path $\gamma'_t$ that supports $G(t)$ does not depend on $t$ while $t\in[a,b]$. So we just need to show that the gain function $\mathfrak{g}(G(t))$ does not change when $t$ passes across $t_i$ or $\tau_j$.

If $t$ passes across $\tau_j$ then either the supporting primitive path does not change or one of its facets $F^{d-1}$ could be replaced by a copy of a sequence $[F^{d-1},G^{d-1},F^{d-1}]$ (or vice versa) for some facets $F^{d-1}$ and $G^{d-1}$ with common primitive $(d-2)$-face. In the latter case gain function $\mathfrak{g}$ does not change because $\mathfrak{g}([F^{d-1},G^{d-1},F^{d-1}])=1.$

If $t$ passes across $t_i$ then either the supporting primitive path does not change or some subpath $[F^{d-1}_{i,1},\ldots,F^{d-1}_{i,2}]$  with each facet containing $F^{d-3}_i$ changes into subpath with the same startpoint and endpoint and again all facets of this new subpath containing $F^{d-3}_i$. In this case the gain function $\mathfrak{g}(G(t))$ will not change due to lemma \ref{local}.

Therefore the gain function $\mathfrak{g}(F(t))$ is constant and $\mathfrak{g}(\gamma_1)=\mathfrak{g}(F(0))=\mathfrak{g}(F(1))=\mathfrak{g}(\gamma_2)$.
\end{proof}

\begin{cor}
The gain function $\mathfrak{g}$ is a homomorphism of the fundamental group $\pi_1(P_\delta)$ into $\mathbb{R}_+$.
\end{cor} 

Since $\mathbb{R}_+$ is a commutative group than we trivially get that $\mathfrak{g}$ also gives us a homomorphism of group $\pi_1(P_\delta)/[\pi_1(P_\delta)]$ (here $[G]$ denotes the commutant of a group $G$) to $\mathbb{R}_+$. This group is isomorphic to a group of homologies $H_1(P_\delta,\mathbb{Z})$ (see \cite{Hat}).

\begin{thm}\label{simpconn}
Given a parallelohedron $P$ with connected $\delta$-surface, if the fundamental group $\pi_1(P_\delta)$ or the group of homologies $H_1(P_\delta, \mathbb Z)$ is trivial then Voronoi's conjecture $\ref{vorcon}$ is true for $P$.
\end{thm}

\begin{proof}
In both cases an arbitrary generic cycle $\gamma$ can be represented as a product $\gamma=\gamma_1\cdot\ldots\cdot\gamma_k$ where $\gamma_i=a_ib_ia_i^{-1}b_i^{-1}$ is a cycle from commutant $[\pi_1(P_\delta)]$. This is true because in both cases $[\pi_1(P_\delta)]=\pi_1(P_\delta)$ (and also this is the trivial group in the first case). It is clear that $\mathfrak{g}(\gamma_i)=1$ and hence $\mathfrak{g}(\gamma)=1$. Thus by lemma \ref{iffcond} Voronoi's conjecture is true for $P$.
\end{proof}

Now we show how to generalize this theorem in terms of $\pi$-surface of a parallelohedron $P$. This theorem can be generalized in two directions. First we can consider polytopes with non-connected $\delta$- or $\pi$-surfaces. In that case the polytope $P$ can be represented as a direct sum of parallelohedra of smaller dimensions as it was proved in \cite{Ord}, so if for both polytopal summands the theorem \ref{simpconn} (and therefore Voronoi's conjecture) is true then it is true for $P$.

The second way of generalization is to introduce new cycles that have gain function 1. We can use a special family of ``half-belt'' cycles on the $\pi$-surface of $P$ for these purposes.

\begin{defin}
The cycle $\gamma$ on  $P_\pi$ is called {\it half-belt cycle} if its support is a combinatorial path $\gamma'=[F_1,F_2,F_3,F_1]$ such that all three facets $F_i$ belong to the same belt of length $6$.

On the $\delta$-surface $P_\delta$ this cycle corresponds to a path that starts on the facet $F_1$ and ends on the opposite facet $F_1'$ of $P$ and crosses only three parallel primitive faces of codimension 2.
\end{defin}

\begin{lem}
For every half-belt cycle $\gamma$ we have $\mathfrak{g}(\gamma)=1$.
\end{lem}

\begin{proof}
Let $\alpha_1\mathbf{n}_1+\alpha_2\mathbf{n}_2+\alpha_3\mathbf{n}_3=\mathbf{0}$ be the unique linear dependence of normal vectors $\mathbf{n}_1,\mathbf{n}_2,\mathbf{n}_3$ to facets $F_1,F_2,$ and $F_3$. Then by definition of the gain function $\mathfrak{g}$ we have 
$$\mathfrak{g}(\gamma)=\mathfrak{g}(F_1,F_2)\cdot\mathfrak{g}(F_2,F_3)\cdot \mathfrak{g}(F_3,F_1')=\frac{|\alpha_2|}{|\alpha_1|}\cdot\frac{|\alpha_3|}{|\alpha_2|}\cdot\frac{|\alpha_1|}{|\alpha_3|}=1.$$
\end{proof}

\begin{thm}\label{vorgener}
If the homology group $H_1(P_\pi, \mathbb Q)$ is generated by half-belt cycles of a parallelohedron $P$ then Voronoi's conjecture is true for $P$.
\end{thm}

\begin{proof} At first we show that $\mathfrak{g}$ is well-defined on $H_1(P_\pi, \mathbb Q)$. As above, we have that the gain function $\mathfrak{g}$ is a homomorphism of $\pi_1(P_\pi)$ to $(\mathbb{R}_+, \cdot)$. The group $(\mathbb{R}_+, \cdot)$ is commutative, so the action $\mathfrak{g}$ is well-defined on the commutant of $\pi_1(P_\pi)$ and therefore on $H_1(P_\pi, \mathbb Z)$. Since $(\mathbb{R}_+, \cdot)$ has no torsion $\mathfrak{g}$ is trivial on the torsion subgroup of $H_1(P_\pi, \mathbb Z)$. Thus $\mathfrak{g}$ is well-defined on $H_1(P_\pi, \mathbb Q)$. 

Further since $\mathfrak{g}$ acts trivially on all generators of $H_1(P_\pi, \mathbb Q)$ it is a trivial action on the group as well. Furthermore as $\mathfrak{g}$ also acts trivially on the torsion subgroup of $H_1(P_\pi, \mathbb Z)$ thus it acts trivially of the group $H_1(P_\pi, \mathbb Z)$. Clearly, standard gluing-map $P_{\delta} \longrightarrow P_{\pi}$ holds action of $\mathfrak{g}$ on cycles of $P_{\delta}$. Thus $\mathfrak{g}$ acts trivially of the group $H_1(P_\delta, \mathbb Z)$ and application of lemma \ref{iffcond} finishes the proof.
\end{proof}

\section{Three-dimensional case}\label{3dim}

In this section we will show that conditions of Theorem \ref{vorgener} hold for all three-dimensional parallelohedra.

\begin{exam}
There are five combinatorial types of three-dimensional parallelohedra. Two types of reducible parallelohedra are cube $\mathcal{C}$ and hexagonal prism $\mathcal{P}$. It easy to see that $\mathcal{C}_\delta$ is a collection of six disjoint open disks (corresponding to facets of $\mathcal{C}$) and $\mathcal{C}_\pi$ is a collection of three disjoint open disks. In both cases it is easy to see that both fundamental groups $\pi_1(\mathcal{C}_\delta)$ and $\pi_1(\mathcal{C}_\pi)$ as well as homology groups $H_1(\mathcal{C}_\delta)$ and $H_1(\mathcal{C}_\pi)$ (over $\mathbb{Z}$ or $\mathbb Q$) are trivial and conditions of both theorems \ref{simpconn} (for non-connected case) and \ref{vorgener} are true. For $\mathcal{P}$ situation is a bit more interesting. The surface $\mathcal{P}_\delta$ is a collection of two open disks (bases of prism) and an open strip (the side surface) and $\mathcal{P}_\pi$ is a collection of a disk and a M\"{o}bius strip. In this case 
$$\pi_1(\mathcal{P}_\delta)\cong\pi_1(\mathcal{P}_\pi)\cong H_1(\mathcal{P}_\delta,\mathbb{Z})\cong H_1(\mathcal{P}_\pi,\mathbb{Z})\cong \mathbb{Z}.$$
It is easy to find generators for all these groups. For both fundamental and homology group of $\delta$-surface the generator is the cycle represented by single belt-cycle on the side surface of $\mathcal{P}$. For $\pi$-surface the generator is represented by unique half-belt cycle of $\mathcal{P}$.

Three types of irreducible parallelohedra are rhombic dodecahedron $\mathcal{R}$, elongated dodecahedron $\mathcal{E}$, and truncated octahedron $\mathcal{O}$ (from left to right in Figure~\ref{3d}).
\begin{figure}[!ht]
\begin{center}
\includegraphics[scale=0.9]{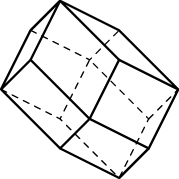}
\hskip 1cm
\includegraphics[scale=0.9]{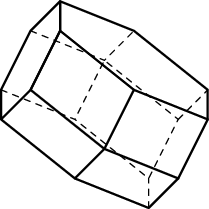}
\hskip 1cm
\includegraphics[scale=0.9]{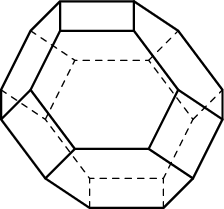}
\end{center}
\caption{Irreducible parallelohedra in $\mathbb{R}^3$.}
\label{3d}
\end{figure}
The rhombic dodecahedron and the truncated octahedron are $2$-primitive (the Zhitomirskii case \cite{Zhit}) since all edges of these polytopes generate belts of length 6. For any such parallelohedron its $\delta$-surface is just the surface of parallelohedron itself. So, $\mathcal{R}_\delta$ and $\mathcal{O}_\delta$ are homeomorphic to the sphere $\mathbb{S}^2$, and $\mathcal{R}_\pi$ and $\mathcal{O}_\pi$ are homeomorphic to the projective plane $\mathbb{RP}^2$. So, groups $\pi_1(\mathcal{R}_\delta)$, $H_1(\mathcal{R}_\delta)$, $\pi_1(\mathcal{O}_\delta)$ and $H_1(\mathcal{O}_\delta)$ are trivial and we can apply theorem \ref{simpconn} for these polytopes. The groups $H_1(\mathcal{R}_\pi,\mathbb{Q})$ and $H_1(\mathcal{O}_\pi,\mathbb{Q})$ are trivial. Thus  for this case 
we can also apply Theorem~\ref{vorgener}.

The most interesting case is the elongated dodecahedron $\mathcal{E}$. The manifold $\mathcal{E}_\delta$ is a sphere with four cuts. (The cuts correspond to four vertical edges of the middle polytope in Figure~\ref{3d}.) The fundamental group $\pi_1(\mathcal{E}_\delta)$ is the free group with 3 generators. We can take three cycles consisting of two half-belts each going around some deleted edge as independent generators. Thus, the homology group $H_1(\mathcal{E}_\delta,\mathbb{Z})$ is isomorphic to $\mathbb{Z}^3$.

The manifold $\mathcal{E}_\pi$ is the real projective plane with two cuts, or equivalently, the M\"obius strip with one cut. One possible triangulation of (closure of) this surface on the figure \ref{pict:rp2}, the cuts are shaded with gray, and opposite points of the boundary circle are glued together as usual.

\begin{figure}[!ht]
\begin{center}
\includegraphics[height=5cm]{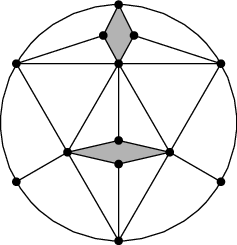}
\caption{A triangulation of $\mathcal{E}_\pi$.}
\label{pict:rp2}
\end{center}
\end{figure}

Fundamental group $\pi_1(\mathcal{E}_\pi)$ is generated by two cycles around cuts and arbitrary half-belt cycle (one can use a cycle represented by half-circle from the figure \ref{pict:rp2} for that) but these three cycles are not independent. The group of one-dimensional homologies also can be found easily using that Betti number $h_1$ of $\mathcal{E}_\pi$ is $2$ (to check this one can calculate the Euler characteristic of $\mathcal{E}_\pi$ using figure \ref{pict:rp2}). For homology groups we have $H_1(\mathcal{E}_\pi,\mathbb{Z})=\mathbb{Z}^2\times T$, where $T$ is the torsion part of homology group and does not affect on existence of canonical scaling. And $H_1(\mathcal{E}_\pi,\mathbb{Q})$ is just $\mathbb{Q}^2$. In both cases generators of non-torsion part can be chosen as one half-belt cycle and composition of two half belt cycles around on of the cuts (or generating cycle of corresponding M\"obius strip and cycle around cut on the M\"obius strip). Thus Theorem \ref{vorgener} can be applied to this case as well.
\end{exam}

\section{Acknowledgments}

Authors would like to thank Professor Nikolai Dolbilin and Professor Robert Erdahl for fruitful discussions and helpful comments and remarks.

The first author is supported by RScF project 14-11-00414.

\end{document}